\documentclass[reqno]{amsart}
\usepackage{amssymb,amsthm,amscd}
\usepackage[pdftex]{graphics}
\input{diagrams}
\newcommand{\rsp}{\raisebox{0em}[2.4ex][1.5ex]{\rule{0em}{2ex} }}
\newcommand{\fa}{{\mathfrak a}}
\newcommand{\fb}{{\mathfrak b}}

\newcommand{\fp}{{\mathfrak p}}

\newcommand{\fq}{{\mathfrak q}}

\newcommand{\ftw}{{\mathfrak 2}}
\newcommand{\Cl}{{\operatorname{Cl}}}
\newcommand{\Gal}{{\operatorname{Gal}}}

\newcommand{\Z}{{\mathbb Z}}
\newcommand{\Q}{{\mathbb Q}}

\newcommand{\cO}{{\mathcal O}}
\newcommand{\eps}{\varepsilon}

\newcommand{\lra}{\longrightarrow}
\newcommand{\la}{\langle}
\newcommand{\ra}{\rangle}

\newcommand{\tomega}{\widetilde{\omega}}

\newtheorem{thm}{Theorem}[section]
\newtheorem{prop}[thm]{Proposition}
\newtheorem{lem}[thm]{Lemma}

\numberwithin{equation}{section}

\title{Harbingers of Artin's Reciprocity Law. \\
          IV. Bernstein's Reciprocity Law}
\author{F. Lemmermeyer}
\email{hb3@ix.urz.uni-heidelberg.de}
\address{M\"orikeweg 1, 73489 Jagstzell, Germany}

\begin{document}

\maketitle

\markboth{Harbingers of Artin's Reciprocity Law}
         {\today \hfil Franz Lemmermeyer}
\begin{center} \today \end{center}
\bigskip

In the last article of this series (see \cite{FB1}) we will first 
explain how Artin's reciprocity law for unramified abelian extensions 
can be formulated with the help of power residue symbols, and then 
show that, in this case, Artin's reciprocity law was already stated 
by Bernstein \cite{Bern2} in the case where the base field contains 
the roots of unity necessary for realizing the Hilbert class field 
as a Kummer extension. Bernstein's article appeared in 1904, almost
20 years before Artin conjectured his version of the reciprocity law,
and seems to have been overlooked completely.

Let me also remark that although we will present Bernstein's 
reciprocity law only for unramified extensions (Takagi created
the general class field theory dealing with ramified abelian 
extensions long after Bernstein's work), the generalization 
to arbitrary abelian extensions of number fields is straightforward.

With hindsight, the basic idea is this: the Artin isomorphism is a 
decomposition law for abelian extensions $K/k$. By adjoining suitable
roots of unity to $k$, the extension $K/k$ will lift to a Kummer
extension. In Kummer extensions, the decomposition of prime ideals
is governed by power residue symbols. We may therefore harbor some
hope of being able to express the content of Artin's reciprocity law
using power residue symbols. Such a description is easy to give
in the case when the base field contains the roots of unity that
are necessary for writing the Hilbert class field $K$ as a Kummer
extension.

Bernstein's reciprocity law apparently did not play any role at
all in the development of Artin's reciprocity law because it has
not been noticed at all. In Sect. \ref{SBRL} we will give (a 
corrected version of) Bernstein's reciprocity law, and in Sect.
\ref{BApp} we will show that this reciprocity law contains
several classical observations on the quadratic and cubic power 
residue characters of quadratic units.

\section{Bernstein's Reciprocity Law}\label{SBRL}

Let us now see how to formulate Artin's reciprocity law with 
the help of power residue symbols. Before we do so we have to 
recall a few basic properties of singular numbers and power
residue symbols.

\subsection*{Power Residue Symbols.}
Let $h \ge 2$ be an integer, and let $k$ be a number field
containing a primitive $h$-th root of unity $\zeta = \zeta_h$. 
For $\alpha \in k^\times$ and prime ideals $\fp \nmid h\alpha$ 
we define the $h$-th power residue symbol $(\alpha/\fp)_h$ by 
demanding that its values are $h$-th roots of unity satisfying 
the congruence
$$ \Big( \frac{\alpha}{\fp} \Big)_h \equiv \alpha^{(N\fp-1)/h} \bmod \fp. $$
Observe that $N\fp \equiv 1 \bmod h$ since $\Q(\zeta) \subseteq k$.

An element $\alpha \in k^\times$ is called {\em singular} if 
$(\alpha) = \fa^h$ is an $h$-th ideal power of some fractional 
ideal $\fa$. It is called {\em primitive} if $k(\sqrt[h]{\alpha}\,)/k$ 
is unramified at all primes dividing $h$. If $\alpha$ is singular and 
primitive, then $k(\sqrt[h]{\alpha})/k$ is unramified at all finite 
primes (if $h > 2$, the extension is automatically unramified at the 
infinite primes because $k$, as an extension of $\Q(\zeta)$, is 
totally complex).

The main observation we will need is the following classical result:

\begin{lem}
If $\alpha \in k^\times$ is singular and primitive, then the
power residue symbol $(\alpha/\fp)_h$ is well defined for all 
prime ideals $\fp \nmid h$.
\end{lem}

In fact, fix a prime ideal $\fp \nmid h$. If $\fp \mid (\alpha)$,
then write $(\alpha) = \fa^h$. Choose an ideal $\fb$ in the ideal
class $[\fa]$ generated by $\fa$ such that $\fb$ is coprime to $\fp$.
Then $\fa = \gamma \fb$. Then $\alpha = \gamma^h \beta$ for some
$\beta \in k^\times$ with $(\beta) = \fb^h$. Clearly 
$(\alpha/\fq)_h = (\beta/\fq)_h$ for all prime ideals 
$\fq \nmid \alpha\beta$, and $(\beta/\fp)_h$ is defined since
$\fp \nmid \beta$ by construction.

\subsection*{Bernstein's Reciprocity Law.}

Now we will formulate Bernstein's version of the reciprocity law.
We will distinguish two cases.

\medskip\noindent{\bf Case I.} $\zeta_h \in k$.

Consider an algebraic number field $k$ whose class group is
cyclic of order $h$, and assume that $k$ contains a primitive
$h$-th root of unity $\zeta$. Then the Hilbert class field of 
$K$ has the form $K = k(\sqrt[h]{\omega}\,)$ for some 
$\omega \in k^\times$.

\begin{thm}\label{TB1}
Let $k$ be a number field with cyclic class group $\Cl(k) = \la c \ra$ 
of order $h$, and assume that $k$ contains the $h$-th roots of unity
$\mu_h = \la \zeta \ra$. Then we can choose $\omega \in k^\times$ 
in such a way that $K = k(\sqrt[h]{\omega}\,)$ is the Hilbert class 
field of $k$, and that 
\begin{equation}\label{EB1}
    \bigg(\frac{\omega}{\fp}\bigg) = \zeta^e \qquad \iff \qquad [\fp] = c^e
\end{equation}
for all prime ideals $\fp \nmid h$, where $(\frac{\cdot}{\fp})$ denotes 
the $h$-th power residue symbol in $k$.
\end{thm}

\begin{proof}
Let $K = k(\sqrt[h]{\omega}\,)$ be the Hilbert class field of $k$.
Let $\fp$ denote a prime ideal in $k$, and let $\sigma$ denote its
Frobenius automorphism. Applying $\sigma$ to $\alpha = \sqrt[h]{\omega}$
we find 
$\alpha^\sigma \equiv \alpha^{N\fp} = \alpha^{N\fp - 1} \alpha
       = \omega^{(N\fp-1)/h} \alpha \equiv 
     (\frac{\omega}{\fp})\ \alpha \bmod \fp$.
Artin's reciprocity law induces an exact sequence
$$ \begin{CD}
    1 @>>> P_k @>>> I_k @>>> \Gal(K/k) @>>> 1,
    \end{CD} $$
where the map $I_k \lra \Gal(K/k)$ is induced by sending a prime ideal $\fp$ to    
its Frobenius automorphism. Since the Artin map only depends on the ideal class
of $\fp$ we get an isomorphism $\Cl(k) \lra \Gal(K/k)$.

Let $\fq$ be a prime ideal in the ideal class $c$; then 
$(\frac{\omega}{\fq}) = \zeta^a$ for some $a$ coprime to $h$. Write 
$ab \equiv 1 \bmod h$; replacing $\omega$ by $\omega^b$ we find that  
$K = k(\sqrt[h]{\omega}\,)$ and $(\frac{\omega}{\fq}) = \zeta$. The 
multiplicativity of the Artin map then guarantees that 
$(\frac{\omega}{\fp}) = \zeta^e$ if and only if $[\fp] \in c^e$.
\end{proof}

Clearly,  Thm. \ref{TB1} is equivalent to Artin's reciprocity law
in unramified abelian extensions of number fields $k$ with cyclic class
group of order $h$ and $\zeta_h \in k$: if (\ref{EB1}) holds, then the
map $\fp \to (\frac{\omega}\fp)$ induces a canonical isomorphism 
$\rho: \Cl(k) \lra \mu_h$ between the ideal class group of $k$ and
the group $\mu_h$ of $h$-th roots of unity. Similarly, mapping
$\sigma \in \Gal(K/k)$ to the root of unity 
$\chi(\sigma) = \sqrt[h]{\omega}^{\sigma-1}$ defines a canonical 
isomorphism $\chi: \Gal(K/k) \lra \mu_h$. Composing $\rho$ with $\chi^{-1}$
then provides us with a canonical isomorphism $\rho$ between the class 
group $\Cl(k)$ and the Galois group $\Gal(K/k)$ of the Hilbert class 
field $K$ of $k$. 

The restriction of Thm. \ref{TB1} to subextensions of the Hilbert class
field can be proved similarly:
 
\begin{thm}\label{TBsp}
Let $\ell$ be a prime number and $k$ a number field whose $\ell$-class
group is cyclic of order $m = \ell^n$. Assume that $k$ contains the
$m$-th roots of unity, and let $K/k$ be the cyclic unramified extension
of degree $m$. 

Let $c$ be an ideal class with order $m$ in $\Cl(k)$. Then we can choose
$\omega \in k^\times$ in such a way that $K = k(\sqrt[m]{\omega}\,)$ and
\begin{equation}\label{EB2}
    \bigg(\frac{\omega}{\fp}\bigg) = \zeta^e
              \qquad \iff \qquad [\fp]^{h/m} = c^e
\end{equation}
for all prime ideals $\fp \nmid h$, where $(\frac{\cdot}{\fp})$ 
denotes the $m$-th power residue symbol in $k$.
\end{thm}

The proof proceeds exactly as above.

\medskip\noindent{\bf Case II.} $\zeta_h \not\in k$.

If $k$ does not contain the roots of unity necessary for defining the
$h$-th power residue symbol, the situation is slightly more involved.
In this case, assume that $\Cl(k)$ is cyclic of order $h$ and set 
$k' = k(\zeta)$ for some primitive $h$-th root of unity. The
Hilbert class field $k^1/k$ becomes a Kummer extension over $k'$, and
we can write $k'k^1 = k'(\sqrt[h]{\omega}\,)$ for a suitable 
$\omega \in k'$. Observe that the degree $(k'(\sqrt[h]{\omega}\,):k)$
divides the class number $h$; if a subextension of $k'/k$ is unramified, 
this degree will be strictly smaller than $h$. 

In fact, set $\nu = (k^1 \cap k':k)$; class field theory predicts that
the group $N_{k'/k}\Cl(k')$ has index $\nu$ in the class group $\Cl(k)$.
We will see below that the description via power residue symbols will
allow us to characterize only those ideal classes that belong to this
subgroup  $N_{k'/k}\Cl(k')$ of index $\nu$ in $\Cl(k)$. 
\smallskip

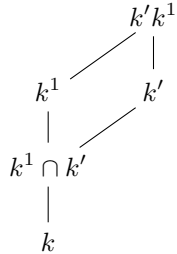
\begin{minipage}{3.5cm}
\begin{diagram}[height=0.5cm,width=0.7cm]
        &         &  k'k^1 \\
        & \ruLine & \dLine  \\
   k^1  &         & k'     \\
 \dLine & \ruLine &        \\
 k^1 \cap k' &    &        \\
 \dLine &         &        \\
  k     &         & 
\end{diagram}
\end{minipage}
\begin{minipage}{8.5cm}
Recall that $k^1/k$ is a class field with conductor $1$ for the ideal
group $P_k$ of principal ideals in $k$. By the Translation Theorem
of class field theory, the extension $k^1k'/k'$ is also a class field
with conductor $1$ for the ideal group 
$$ H_{k'} = \{\fa' \in I_{k'}: N_{k'/k} \fa' \in P_k\} $$
consisting of all ideals whose norms down to $k$ are principal.
\end{minipage}
\smallskip

For formulating (a piece of) Artin's reciprocity law using power residue
symbols we use the well known transfer formula:

\begin{lem} 
Let $L/K$ be a finite abelian extension of number fields.
If $F/K$ is a finite extension, then $FL/F$ is abelian, and we have,
for all prime ideals $\fp$ unramified in $F$,  
\begin{equation}\label{Atrans}
        \Big(\frac{FL/F}{\fq} \Big)\Big|_L 
            =  \Big(\frac{L/K}{N_{F/K}\fq} \Big)
            = \Big(\frac{L/K}{\fp} \Big)^f,
\end{equation}
where $\fq$ is a prime ideal in $F$ above $\fp$, and where
$f = f(\fq|\fp)$ is the inertia degree of $\fq$. If, in particular,
$F/K$ is a subextension of $L/K$, then 
$$ \Big(\frac{L/F}{\fq} \Big) 
            =  \Big(\frac{L/K}{N_{F/K}\fq} \Big)
            = \Big(\frac{L/K}{\fp} \Big)^f. $$
\end{lem}

In our case, (\ref{Atrans}) says that 
$$ \Big(\frac{k^1k'/k'}{\fa'} \Big)\Big|_k 
    = \Big(\frac{k^1/k}{N_{k'/k}\fa'} \Big). $$
Let $c$ be an ideal class generating the cyclic group $\Cl(k)$; then $c^\nu$
is a norm of some class $C \in \Cl(k')$, and $c^\nu$ generates the subgroup
$N_{k'/k} \Cl(k')$ of $\Cl(k)$. Let $\fp'$ be a prime ideal in $C$; as in 
Case~I we can choose $\omega \in k'$ in such a way that 
$(\frac{\omega}{\fp'}) = \zeta^\nu$, where $\zeta$ is a primitive $h$-th 
root of unity\footnote{Observe that $\omega$ is a $\nu$-th power,
so we only get $h/\nu$-th roots of unity on both sides.}. Thus we get

\begin{thm}[Bernstein's Reciprocity Law]
Let $k$ be a number field whose class group $\Cl(k) =  \la c \ra$ is cyclic
of order $h$, and let $k' = k(\zeta)$, where $\zeta$ is a primitive $h$-th 
root of unity. Then we can choose $\omega \in k'$ in such a way that the
Hilbert class field $K$ of $k$ lifts to the Kummer extension 
$Kk' = k'(\sqrt[h]{\omega}\,)$, and that for prime ideals $\fp'$ in $k'$
we have 
\begin{equation}\label{EB3}
 \Big(\frac{\omega}{\fp'}\Big) = \zeta^e 
             \quad \iff \quad  N_{k'/k} \fp' \in c^e
\end{equation}
for all prime ideals $\fp \nmid h$, where $(\frac{\omega}{\cdot})$ 
denotes the $h$-th power residue symbol in $k'$.
\end{thm}

The extension of Bernstein's reciprocity law to number fields whose class
groups are not necessarily cyclic is purely formal:

\begin{thm}[Bernstein's Reciprocity Law]
Let $k$ be a number field whose class group $\Cl(k) =  \la c_1, \ldots, c_q \ra$
is the direct sum of groups $C_j = \la c_j \ra$ of prime power order $h_j$
(then the class number $h$ of $k$ is given by $h = h_1 \cdots h_q$).
Let $\zeta_j$ denote a primitive $h_j$-th root of unity, and set 
$k' = k(\zeta_1, \ldots, \zeta_q)$.

Then there exist elements $\omega_1, \ldots, \omega_q \in k'$
such that $k'(\sqrt[h_1]{\omega_1}, \ldots, \sqrt[h_q]{\omega_q}\,)$
is the compositum of $k'$ and the Hilbert class field of $k$.
Moreover, the $\omega_j$ can be chosen in such a way that
$$ \bigg(\frac{\omega_j}{\fp'}\bigg) = \zeta_j^{e_j}
 \text{for } j=1, \ldots, q \quad \text{if and only if} \quad
   [N_{k'/k}\fp'] = c_1^{e_1} \cdots c_q^{e_q}$$
for all prime ideals $\fp \nmid h$, where $(\frac{\,\cdot\,}{\cdot})$ 
denotes the $h$-th power residue symbol in $k'$.
\end{thm}

The connection between explicit reciprocity laws for power residue 
symbols and Artin's reciprocity law was used by Artin himself to
prove special cases of his reciprocity law in \cite{ArtL}. Artin
was able to prove his reciprocity law for general abelian extensions
only after Chebotarev provided the key idea in his proof what 
became known as Chebotarev's density theorem.

Bernstein's formulation of his reciprocity law in \cite{Bern2}
is only correct for number fields containing the appropriate
roots of unity:

\medskip\noindent{\bf Bernstein's Formulation.}
{\em Let $k$ be a number field with class number $h$. Decompose the
class group $\Cl(k) = C_1 \oplus \cdots \oplus C_q$ into groups
$C_j$ of prime power order $h_j$ (then $h = h_1 \cdots h_q$), and
pick generators $c_j$ of the $C_j$. Let $\zeta_j$ be a primitive
$h_j$-th root of unity, let $k' = k(\zeta_1, \ldots, \zeta_q)$,
and put $r = (k':k)$.

Then there exist elements $\omega_1, \ldots, \omega_q \in k'$
such that $k'(\sqrt[h_1]{\omega_1}, \ldots, \sqrt[h_q]{\omega_q}\,)$
is the compositum of $k'$ and the Hilbert class field of $k$.
Moreover, the $\omega_j$ can be chosen in such a way that the
relations 
$$ \bigg(\frac{\omega_1^{(\rho)}}{\fp'}\bigg) = \zeta_1^{e_{1,\rho}},  \ldots,
   \bigg(\frac{\omega_q^{(\rho)}}{\fp'}\bigg) = \zeta_q^{e_{q,\rho}} $$
are the necessary and sufficient conditions for
$$ {} [N_{k'/k}\fp'] = c_1^{e_1} \cdots c_q^{e_q}, $$
where}
$$ e_1 = \sum_{\rho=1}^r e_{1,\rho}, \ \ldots, \
   e_q = \sum_{\rho=1}^r e_{q,\rho}. $$

\subsection{Power Residue Symbols}
It is clear that the right hand side in the correspondence (\ref{EB3})
does not change if we replace $\fp'$ by one of its conjugates. Let us
therefore convince ourselves that the left hand side is also invariant.

To this end, assume that $(\frac{\omega}{\fp'}) = \zeta^e$ and let
$\sigma$ denote an automorphism of $Kk'/k'$. Then 
$k'(\sqrt[h]{\omega^\sigma}\,) = k'(\sqrt[h]{\omega}\,)$ implies that
$\omega^\sigma = \alpha^h\omega^a$ for $\alpha \in k'$ and some exponent
$a$ coprime to $h$. Since $k^1/k$ is abelian, it follows from Kummer
theory (see e.g. \cite[Satz 147]{HZB} or \cite[Lemma 14.7]{Wash}) 
that we also must have $\zeta^\sigma = \zeta^a$ for the same value of $a$. 

The basic formalism of power residue symbols (see \cite[Ch. 4]{LRL}) 
now shows
$$ \Big(\frac{\omega}{\fp'}\Big)^\sigma 
       = \Big(\frac{\omega^\sigma }{{\fp'}^\sigma }\Big)
       = \Big(\frac{\omega^a}{{\fp'}^\sigma }\Big)
       = \Big(\frac{\omega}{{\fp'}^\sigma }\Big)^a
       = \Big(\frac{\omega}{{\fp'}^\sigma }\Big)^\sigma, $$
hence
\begin{equation}\label{EBprs}
 \Big(\frac{\omega}{\fp'}\Big) = \Big(\frac{\omega}{{\fp'}^\sigma }\Big) 
\end{equation}
as expected. 

Now let $\tau = \sigma^{-1}$; then 
$$  \Big(\frac{\omega}{\fp'}\Big)^\sigma 
     = \Big(\frac{\omega}{{\fp'}^\tau}\Big)^\sigma 
     = \Big(\frac{\omega^\sigma}{\fp'}\Big). $$
This implies that if $(\omega/\fp') = \zeta^e$, then
$$ \prod_{\sigma \in \Gal(Kk'/k')} \Big(\frac{\omega^\sigma}{\fp'}\Big)
   = \prod_{\sigma \in \Gal(Kk'/k')} \Big(\frac{\omega}{\fp'}\Big)^\sigma
   = N_{k'/k}(\zeta^e) = 1. $$
  
This means that in the special case where $\Cl(k)$ is cyclic or
prime order $p$ and where $k$ does not contain the $p$-th roots 
of unity, the product of the power residue symbols in
Bernstein's formulation is always trivial, that is, we have
$e_1 = 0$. This shows that Bernstein's formulation of his
reciprocity law is incorrect in the case where roots of unity
have to be adjoined.

\section{Power Residue Characters of Quadratic Units}\label{BApp}

The corrected formulation of Bernstein's reciprocity law is 
much more than a twisted version of Artin's reciprocity law;
in fact, Bernstein's reciprocity law contains many results
on the power residue characters of quadratic units obtained by
Dirichlet, Kronecker, Scholz, Aigner and a host of other
mathematicians in the wake of Emma Lehmer's work on these topics
in the 1970s. 

In this section we will be content with giving a few examples
that show how to apply Bernstein's reciprocity law to such problems. 
This approach can be found in various articles on the power residue
characters of units, for example in Halter-Koch \cite{HaKo}.

Before we begin let us recall the relevant notation. Let $k$ be a
number field containing the $m$-th roots of unity, and let $\eps$
be a unit in the ring of integers $\cO_k$. Define the $m$-th power
residue symbol $(\alpha/\fp)_m$ for prime ideals $\fp$ in $\cO_k$
coprime to $m\alpha$ by $(\alpha/\fp)_m = \zeta_m^r$ if 
$\alpha^{(N\fp-1)/m} \equiv \zeta_m^r \bmod \fp$ (observe that
$N\fp \equiv 1 \bmod m$ since $k$ contains the $m$-th roots of unity).
If the symbol $(\eps/\fp)_m$ does not depend on the choice of the 
prime ideal $\fp$ above $p$, then we set $(\eps/p)_m = (\eps/\fp)_m$.

\subsection{Dirichlet}
The following example due to Dirichlet shows that Bernstein's reciprocity
law is weaker than Artin's if $\nu \ne 1$.
\smallskip

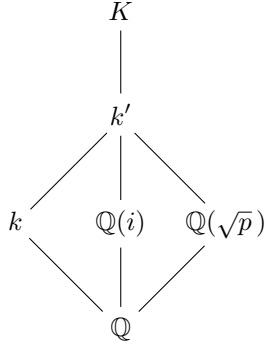
\begin{minipage}{4.5cm}
\begin{diagram}[height=0.7cm,width=0.7cm]
   &         &    K   &         &   \\
   &         & \dLine &         &   \\
   &         &    k'  &         &   \\
   & \ruLine & \dLine & \luLine &   \\
 k &         & \Q(i)  &         & \Q(\sqrt{p}\,) \\
   & \rdLine & \dLine & \ruLine &   \\
   &         &  \Q    &         &   
\end{diagram}
\end{minipage}
\begin{minipage}{7.5cm}
Let $p \equiv 1 \bmod 8$ be prime, write $p = a^2 + 4b^2$ and let $\eps$
denote the fundamental unit of $\Q(\sqrt{p}\,)$. The $2$-class group
of $k = \Q(\sqrt{-p}\,)$ is cyclic of order divisible by $4$. The unique
cyclic quartic unramified extension $K/k$ contains the quadratic
subextension $k' = k(i)$; this implies $K' = K$, hence the extension
$K'/k'$ is quadratic, and we have $\nu = 2$ in this case.

We will see below that $K = k(\sqrt{\eps}\,)$, and that $K/k$ is a cyclic
quartic extension. The extension $L/\Q'$ with $\Q' = \Q(i)$, on the other 
hand, is a biquadratic extension since $L = \Q'(\sqrt{a+2bi},\sqrt{a-2bi}\,)$.
\end{minipage}
\smallskip

Applying Bernstein's reciprocity law (more exactly the special case 
Thm. \ref{TBsp}) to this situation we find

\begin{prop}
Let $p \equiv 1 \bmod 8$ be prime, write $p = a^2 + 4b^2$, let 
$\eps > 1$ denote the fundamental unit of $\Q(\sqrt{p}\,)$, and let
$h$ be the class number of $k = \Q(\sqrt{-p}\,)$. Then 
$k(\sqrt{\eps}\,)/k$ is a cyclic quartic unramified extension of $k$
containing $k' = k(i)$ as its quadratic subextension, and we have
$$ \Big(\frac{\eps}{q}\Big) = \begin{cases}
                +1 & \text{ if } q^{h/4} = x^2 + py^2, \\
                -1 & \text{ if } q^{h/4} = 2x^2 + 2xy + \frac{p+1}2 y^2. 
                              \end{cases} $$
\end{prop}

\begin{proof}
Recall (see \cite{Bran} or \cite[Ex. 5.10]{LRL}) that $(a+2bi)\eps$
is a square in $\Q(i,\sqrt{p}\,)$; this shows that 
$K = k'(\sqrt{\eps}\,) = k'(\sqrt{a+2bi}\,)$. Since $b$ is even, the
extension $K/k'$ is unramified at the primes above $2$, and the fact
that $k'(\sqrt{a+2bi}\,) =  k'(\sqrt{a-2bi}\,)$ shows that primes above
$p$ cannot ramify either. Thus $K'/k$ is unramified everywhere.

The $2$-class group $\Cl_2(k)$ is cyclic by Gauss's genus theory. The
unique element of order $2$ in the ideal class group of $k$ is represented
by the ideal $\ftw = (2,1+\sqrt{-p}\,)$. If $\fq$ is a prime ideal with
odd prime norm  $q$ in the class $[\ftw]$, then $\ftw \fp$ is principal,
and taking norms shows that $2q = w^2 + py^2$. Since $q$ and $y$ must be
odd, we can set $w = 2x+y$ and find, after cancelling the common factor
$2$ from both sides, that $q = 2x^2 + 2xy + \frac{p+1}2 y^2$. 

For applying Bernstein's reciprocity law observe that 
$K = k'(\sqrt[4]{\eps^2}\,)$, and that the fourth power residue
symbols $(\eps^2/\fp)_4$ are just quadratic residue symbols $(\eps/\fp)$.
\end{proof}

\medskip\noindent{\bf Example.}
Consider the number field $k = \Q(\sqrt{-17}\,)$. Its class group 
$\Cl(k)$ is cyclic of order $4$ and is generated by the class of 
the prime ideal $\fq = (3,1+\sqrt{-17},)$. Observe that 
$\eps = 4+\sqrt{17} = \eta^2$ for 
$\eta = \frac{\sqrt{1+4i} + \sqrt{1-4i}}{1-i}$.

Scholz's reciprocity law (see \cite[Ch. 5]{LRL}) shows that
$(\eps/q) = (q/17)_4(17/q)_4$ for primes $q \equiv 1 \bmod 4$
with $(\frac{17}{q}) = 1$, hence $(\eps/q) = +1$ if $q = x^2 + 17y^2$,
and $(\eps/q) = -1$ if $q = 2x^2 + 2xy + 9y^2$. This result is due
to Dirichlet \cite[\S\ 4]{Diri} and was rediscovered by Brandler
(see \cite{Bran} and \cite[Ch. 5]{LRL}).

\subsection{Kronecker}
Let $a$ be an odd positive integer, and assume that $m = 3a^2 \pm 4$
is squarefree. Then 
$$ \eps = \frac13\Big(\frac{3a + \sqrt{3m}}2\,\Big)^2 
        = \pm 1 + 3a \frac{a + \sqrt{3m}}2 $$
is a unit in the ring of integers of $k = \Q(\sqrt{3m}\,)$.
If $a = 3b$ and $m = 27b^2 \pm 4$, then $K' = k'(\sqrt[3]{\eps}\,)$
is a cubic unramified Kummer extension of $k' = \Q(\sqrt{m},\sqrt{-3}\,)$.
In fact since $(k':k) = 2$, the unit $\eps$ is a cube in $k'$ if and 
only if it is a cube in $k$, which in turn is equivalent to 
$3(3a+\sqrt{3m}\,)$ being a cube in $\cO_k = \Z[\frac{1+\sqrt{3m}}2\,]$. 
The equation $9a + 3\sqrt{3m} = \big(\frac{r+s\sqrt{3m}}2 \big)^3$
leads to 
$$ 9a = r(r^2 + 9rs^2m) \quad \text{and} \quad
    3 = 3s(3r^2 + ms^2), $$
which is easily seen to be impossible.

The cyclic unramified extension $K'/k'$ descends to $F =\Q(\sqrt{-m}\,)$
in the sense that the abelian extension $K'/F$ contains a cyclic cubic 
unramified subextension $L/F$. We can construct this extension
explicitly be setting $\theta = \sqrt[3]{\eps} + \sqrt[3]{\eps'}$;
in fact we find
$$ \theta^3 = \eps + \eps' + 
         3\sqrt[3]{\eps\eps'}( \sqrt[3]{\eps} + \sqrt[3]{\eps'})
           = 27b^2 \pm 2 + 3\theta, $$
hence $L$ is the compositum of $F$ and the cubic extension
generated by a root $\theta$ of the polynomial $x^3 - 3x - 27b^2 \mp 2$.
   
Here are a few small examples of odd values of $b$ for which 
$m = 27b^2 - 4$ is squarefree; in the table below, $h$ denotes
the class number of the quadratic number field $\Q(\sqrt{-m}\,)$.
$$ \begin{array}{c|rrrrr}
    b &   1 &    3 &    5 &     7 &     9  \\ \hline  
   -m & -23 & -239 & -671 & -1319 & -2183  \\
    h &   3 &   15 &   30 &    45 &    42
   \end{array} $$

Here is a similar table for $m = 27b^2 + 4$:
$$ \begin{array}{c|rrrrr}
    b &   1 &    3 &    5 &     7 &     9  \\ \hline  
   -m & -31 & -247 & -679 & -1327 & -2191  \\ 
    h &   3 &    6 &   18 &    15 &    30 
   \end{array} $$

Applying Bernstein's reciprocity law to the two examples $b = 1$
with class number $3$ we find

\begin{prop}
Let $\eps = \frac{25+3\sqrt{69}}2$ be the fundamental unit of
$k = \Q(\sqrt{69}\,)$. Then 
$$ \Big(\frac{\eps}{p}\Big)_3 = 1 \quad \iff \quad
     p = x^2 + xy + 6y^2 $$
for primes $p$ with $(\frac{-3}p) = (\frac{-23}p) = 1$.

Similarly, let $\eta = \frac{29+3\sqrt{93}}2$ be the fundamental unit of
$k = \Q(\sqrt{93}\,)$. Then 
$$ \Big(\frac{\eps}{p}\Big)_3 = 1 \quad \iff \quad
     p = x^2 + xy + 8y^2 $$
for primes $p$ with $(\frac{-3}p) = (\frac{-31}p) = 1$.
\end{prop}

The last example was essentially discovered by Kronecker, who
considered in \cite{Kron} the splitting field $L$ of the
polynomial $f(x) = (x^3-10x)^2 + 31(x^2-1)^2$. This number field $L$ 
has degree $6$ and Galois group $S_3$; it is the compositum of the
complex quadratic number field $\Q(\sqrt{-31}\,)$ and the cubic field 
with discriminant $-31$ generated by a root of $x^3 + 11x^2 + 38x + 31$.
Write $\omega = \frac{-1+\sqrt{-3}}2$ and $\tomega = \frac{-1+\sqrt{-31}}2$;
then $\eta = 1 - \tomega + 3\omega = \frac12(3\sqrt{-3} + \sqrt{-31}\,)$
is a unit in $L' = L(\sqrt{-3}\,)$, and $\eta^2 = \frac12(-29+\sqrt{93}\,)$ 
actually shows that $\eta$ is the fundamental unit in 
$\Q(\sqrt{-3},\sqrt{-31}\,)$. 

Now prime ideals $\fp$ in $k$ above primes $p$ with $(\frac{-31}p) = +1$ 
split in completely in the Hilbert class field of $k$ if and only if $\fp$ 
is principal, which is equivalent to $p$ being represented by the
principal form $x^2 + xy + 8y^2$. Since $L' = k'(\sqrt[3]{\eps}\,)$,
Bernstein's reciprocity law tells us that
$$ \Big(\frac{\eps}{p}\Big)_3 = 1  \quad \iff \quad
    p = x^2 + xy + 8y^2 $$
for primes $p$ with $(\frac{-31}p) = (\frac{-3}p) = +1$.

A general result containing Kronecker's example as a special case
can be found e.g. in Weinberger \cite{Weinb}.

\subsection{Quartic Character of Certain Quadratic Units}

Let $p \equiv 5 \bmod 8$ and $q \equiv 3 \bmod 4$ be primes such 
that $(\frac pq) = +1$. The $2$-class group of the complex quadratic 
number field $k = \Q(\sqrt{-pq}\,)$ is cyclic of order divisible by $4$,
hence $k$ admits a cyclic quartic unramified extension $K/k$. Over
$k' = k(i)$, this extension can be realized as a Kummer extension:

\begin{prop}
Let $\eps$ denote the fundamental unit of $F = \Q(\sqrt{pq}\,)$, where $p$
and $q$ are as above. Then $\eps \equiv s \bmod 4$ for $s \in \{\pm 1\}$,
and $K' = k'(\sqrt[4]{\eps}\,)$. 
\end{prop}

\begin{proof}
Write $\eps = T + U \sqrt{pq}$; from $T^2 - pqU^2 = 1$ it follows
that $T$ is odd and $U \equiv 0 \bmod 4$, hence $\eps \equiv \pm 1 \bmod 4$.
By a routine calculation one verifies that $k'(\sqrt[4]{\eps}\,)/k$ is a
cyclic quartic unramified extension, from which the claim follows
since $k$ has a cyclic $2$-class group.
\end{proof}

The $2$-class group $\Cl_2(k)$ of $k = \Q(\sqrt{-pq}\,)$, where 
$p \equiv 5 \bmod 8$ and $q \equiv 3 \bmod 4$ are prime, is cyclic
of order divisible by $4$. The principal class of in the form class
group with discriminant $-pq$ is represented by $Q_0(x,y) = x^2 + xy + my^2$
with $m = \frac{pq+1}4$, and the unique class of order $2$ by 
$Q_1(x,y) = qx^2 + qxy + ny^2$ with $n = \frac{p+q}4$. Bernstein's
reciprocity law applied to $k$ now immediately gives the following

\begin{prop}
Let $p \equiv 5 \bmod 8$ and $q \equiv 3 \bmod 4$ be primes such that
$(\frac pq) = +1$, and let $\eps$ denote the fundamental unit of
$F = \Q(\sqrt{pq}\,)$. Choose $s \in \{ \pm 1\}$ such that
$s\eps \equiv 1 \bmod 4$, and set $pq = 4m-1$ and $p+q = 4n$. Then
$$ \Big(\frac{s\eps}\ell\Big)_4 = \begin{cases}
               +1 & \text{ if } \ell^{h/4} =  x^2 + xy + my^2, \\
               -1 & \text{ if } \ell^{h/4} = qx^2 + qxy + ny^2             
                         \end{cases} $$
for all primes $\ell$ that split in $F' = \Q(i,\sqrt{pq}\,)$, 
where $h$ denotes the class number of $k = \Q(\sqrt{-pq}\,)$.
\end{prop}

From $\ell \equiv 1 \bmod 4$ we see that $x$ must be odd and $y = 2z$
must be even; thus 
\begin{align*}
  x^2 +  xy + my^2 & = x^2 + 2xz + 4mz^2 = (x+z)^2 + pqz^2, \\
 qx^2 + qxy + ny^2 & = qx^2 + 2qxz + 4nz^2 = q(x+z)^2 + pz^2,
\end{align*}
hence our result can also be stated in the form
$$  \Big(\frac{s\eps}\ell\Big)_4 = \begin{cases}
               +1 & \text{ if } \ell^{h/4} =  x^2 + pqy^2, \\
               -1 & \text{ if } \ell^{h/4} = qx^2 +  py^2.             
                         \end{cases} $$

\medskip\noindent{\bf Example 1.}
In the simplest example $p = 13$, $q = 3$ we find $\eps = 25 + 4 \sqrt{39}$,
$s = 1$ and $h = 4$, hence
$$  \Big(\frac{\eps}\ell\Big)_4 = \begin{cases}
               +1 & \text{ if } \ell =  x^2 + 39y^2, \\
               -1 & \text{ if } \ell = 3x^2 + 13y^2.             
                         \end{cases} $$
Since $\ell \equiv 1 \bmod 4$ we see that $2 \mid y$ in the first and
$2 \mid x$ in the second case.

The first few primes $\ell \equiv 1 \bmod 4$ represented by 
$Q_1(x,y) = x^2 + 39 y^2$ and $Q_2(x,y) = 3x^2 + 13y^2$ are
$$ \begin{array}{c|cccccc}
  \rsp \ell  &    61    &    157   &    181   & 277       & 313 & 337 \\ \hline
  \rsp   Q   & Q_2(4,1) & Q_1(1,2) &  Q_1(5,4) & Q_1(11,2) & Q_2(10,1) 
             & Q_2(2,5) \\
   (\eps/\ell)_4 & -1 & +1 & +1 & +1 & -1 & -1    
   \end{array} $$

\medskip\noindent{\bf Example 2.}
If $p = 37$ and $q = 3$, then $\eps = 295 + 28 \sqrt{111}$, hence $s = -1$.
Here we have $h = 8$, and we find the following results:
$$  \Big(\frac{-\eps}\ell\Big)_4 = \begin{cases}
               +1 & \text{ if } \ell^2 =  x^2 + 111y^2, \\
               -1 & \text{ if } \ell^2 = 3x^2 + 37y^2.             
                         \end{cases} $$
Composition of forms shows that this result is equivalent to 
$$  \Big(\frac{-\eps}\ell\Big)_4 = \begin{cases}
               +1 & \text{ if } \ell =  x^2 + 111y^2, 3x^2 + 37y^2 \\
               -1 & \text{ if } \ell^2 = 4x^2 \pm xy + 7y^2.             
                         \end{cases} $$
The first few primes $\ell \equiv 1 \bmod 4$ represented by
$Q_1(x,y) = x^2 + 111y^2$, $Q_2(x,y) = 3x^2 + 37y^2$,  
$Q_4 = 4x^2 + xy + 7y^2$ and $Q_4' = 4x^2 - xy + 7y^2$ are
$$ \begin{array}{c|ccccc}
  \rsp \ell           &  73   & 157  &  181  &  229   &  337 \\ \hline
  \rsp  Q             & Q_4'(2,3) & Q_4(6,1) & Q_4(2,5) & Q_2(8,1) & Q_2(10,1) \\
  \rsp (-\eps/\ell)_4 &   -1  & -1 & -1 & +1 & +1 
   \end{array} $$

\bigskip


\begin{thebibliography}{99}

\bibitem{ArtL} E. Artin,
{\em \"Uber eine neue Art von L-Reihen},
Abh. Semin. Hamburg {\bf 3} (1924), 89--108
%

\bibitem{Bern1} F. Bernstein,
{\em \"Uber den Klassenk\"orper eines algebraischen Zahlk\"orpers},
Erste u. zweite Mitteilung, G\"ott. Nachr. (1903), 46--58; 304--311
%

\bibitem{Bern2} F. Bernstein,
{\em \"Uber unverzweigte Abelsche K\"orper (Klassenk\"orper)
     in einem imagin\"aren Grund\-be\-reich},
     Jahresber. DMV {\bf 13} (1904), 116--119
%

\bibitem{Bran} J. Brandler,
{\em Residuacity properties of real quadratic units},
J. Number Theory {\bf 5} (1973), 271--287
%

\bibitem{Diri} P.G.L. Dirichlet,
{\em Untersuchungen \"uber die Theorie der quadratischen Formen},
Abh. K\"onigl. Preuss. Akad. Wiss. 1833, 101--121; Werke I, 195--218
%

\bibitem{HaKo} F. Halter-Koch,
{\em Konstruktion von Klassenk\"orpern und Potenzrestkriterien
     f\"ur quadratische Einheiten}, 
manuscr. math. {\bf  54} (1986), 453--492
%

\bibitem{HZB} D. Hilbert,
{\em Die Theorie der algebraischen Zahlk\"orper},
Jahresber. DMV 1897, 175--546; Gesammelte Abh. I, 63--363;
Engl. Transl. by I. Adamson, Springer-Verlag 1998
%

\bibitem{Kron} L. Kronecker,
{\em Ueber die Potenzreste gewisser complexer Zahlen},
Monatsber. Berlin (1880), 404--407; Werke II, 95--101
%


\bibitem{LRL} F. Lemmermeyer,
{\em Reciprocity Laws. From Euler to Eisenstein},
Springer-Verlag 2000
%

\bibitem{FB1} F. Lemmermeyer,
{\em Harbingers of Artin's Reciprocity Law}, I, II, III
%

\bibitem{Wash} L. Washington,
{\em Introduction to Cyclotomic Fields}, 
Springer-Verlag 1982
%

\bibitem{Weinb} P. Weinberger,
{\em The cubic character of quadratic units},
Proc. of the Number Theory  Conf. Boulder, 
Colorado (1972), 241--242
%

\end{thebibliography}
\end{document}